\documentclass{article}

\usepackage[english]{babel}
\usepackage{amssymb}
\usepackage{amsmath}
\usepackage{amsthm}

\newtheorem{theorem}{Theorem}[section]
\newtheorem{definition}[theorem]{Definition}
\newtheorem{lemma}[theorem]{Lemma}
\newtheorem{proposition}[theorem]{Proposition}
\newtheorem{corollary}[theorem]{Corollary}
\newtheorem{remark}[theorem]{Remark}

\newtheorem{example}[theorem]{Example}

\newcommand{\hh}{{\mathbb{H}}}

\newcommand{\cc}{{\mathbb{C}}}
\newcommand{\rr}{{\mathbb{R}}}
\newcommand{\nn}{{\mathbb{N}}}
\newcommand{\s}{{\mathbb{S}}}
\newcommand{\z}{{\mathcal{Z}}}
\newcommand{\f}{{\mathcal{F}}}

\newcommand{\mr}{{\mathcal{M}}}

\newcommand{\rft}{{\mathfrak{G}}}

\newcommand{\I}{{\mathrm{Id}}}
\newcommand{\B}{{\mathbb{B}}}

\newcommand{\reg}{\frak{Reg}(\B,\B)}

\title{\bf Regular vs. classical\\ M\"obius transformations of the quaternionic unit ball}
\author{Cinzia Bisi\footnote{Partially supported by GNSAGA of the INdAM and by FIRB ``Geometria Differenziale Complessa e Dinamica Olomorfa''.}\\
\normalsize Universit\`a degli Studi di Ferrara\\
\normalsize Dipartimento di Matematica e Informatica\\
\normalsize Via Machiavelli 35, 44121 Ferrara, Italy\\
\normalsize cinzia.bisi@unife.it
\and
Caterina Stoppato$^*$ \footnote{Partially supported by FSE and by Regione Lombardia.}\\ 
\normalsize Universit\`a degli Studi di Milano\\
\normalsize Dipartimento di Matematica ``F. Enriques''\\
\normalsize Via Saldini 50, 20133 Milano, Italy\\  
\normalsize caterina.stoppato@unimi.it\\}

\date{  }

\begin{document}
\maketitle


\begin{abstract}
The regular fractional transformations of the extended quaternionic space have been recently introduced as variants of the classical linear fractional transformations. These variants have the advantage of being included in the class of slice regular functions, introduced by Gentili and Struppa in 2006, so that they can be studied with the useful tools available in this theory. We first consider their general properties, then focus on the regular M\"obius transformations of the quaternionic unit ball $\B$, comparing the latter with their classical analogs. In particular we study the relation between the regular M\"obius transformations and the Poincar\'e metric of $\B$, which is preserved by the classical M\"obius transformations. Furthermore, we announce a result that is a quaternionic analog of the Schwarz-Pick lemma.
\end{abstract}

\vfill
\section{Classical M\"obius transformations and\\ Poincar\'e distance on the quaternionic unit ball} \label{classicM}

A classical topic in quaternionic analysis is the study of M\"obius transformations. It is well known that the set of \emph{linear fractional transformations} of the extended quaternionic space $\hh \cup \{\infty\} \cong \mathbb{HP}^1$
\begin{equation}\label{IndianaM}
\mathbb{G}=\left\{g(q)=(aq+b)\cdot (cq+d)^{-1} \ \left|\ \begin{bmatrix}
a & b \\
c & d
\end{bmatrix} \in GL(2,\hh) \right\}\right.
\end{equation}
is a group with respect to the composition operation. We recall that $GL(2,\hh)$ denotes the group of $2 \times 2$ invertible quaternionic matrices, and that $SL(2,\hh)$ denotes the subgroup of those such matrices which have unit Dieudonn\'e determinant (for details, see \cite{poincare} and references therein). It is known in literature that $\mathbb{G}$ is isomorphic to $PSL(2,\hh)= SL(2,\hh) / \{\pm \I\}$ and that all of its elements are conformal maps. Among the works that treat this matter, even in the more general context of Clifford Algebras, let us mention \cite{ahlforsmoebius, maass, vahlen}. 
The alternative representation
\begin{equation}\label{InvM}
\mathbb{G}=\left\{F_A(q)=(qc+d)^{-1} \cdot (qa+b) \ \left|\ A=\begin{bmatrix}
a & c \\
b & d
\end{bmatrix} \in GL(2,\hh) \right\}\right.
\end{equation}   
is also possible, and the anti-homomorphism $GL(2,\hh) \to \mathbb{G}$ mapping $A$ to $F_A$ induces an anti-isomorphism between $PSL(2,\hh)$ and $\mathbb{G}$.
The group $\mathbb{G}$ is generated by the following four types of transformations:
\begin{itemize}
\item[(i)] $L_1(q)=q+b,\,\,\, b \in \mathbb{H};$
\item[(ii)] $L_2(q)=q \cdot a, \,\,\, a \in \mathbb{H}, |a|=1;$
\item[(iii)] $L_3(q)= r \cdot q = q \cdot r, \,\,\, r \in \mathbb{R}^+ \setminus \{ 0 \};$
\item[(iv)] $L_4(q)=q^{-1}.$
\end{itemize}
Moreover, if $\mathcal{S}_i$ is the family of all real $i$-dimensional spheres, if $\mathcal{P}_i$ is the family of all real $i$-dimensional affine subspaces of $\mathbb{H}$ and if $\mathcal{F}_i=\mathcal{S}_i \cup \mathcal{P}_i$ then $\mathbb{G}$ maps elements of $\mathcal{F}_i$ onto elements of $\mathcal{F}_i,$ for $i=3,2,1$. At this regard, see \cite{wilker}; detailed proofs of all these facts can be found in \cite{poincare}.

The subgroup $\mathbb{M}\leq\mathbb{G}$ of \emph{(classical) M\"obius transformations} mapping the quaternionic open unit ball 
$$\mathbb{B}=\{q \in \hh\ |\  |q|<1\}$$ 
onto itself, has also been studied in detail. Let us denote $H = 
\begin{bmatrix}
1 & 0 \\
0 & -1
\end{bmatrix}$ and 
$$Sp(1,1) = \{C \in GL(2, \hh) \ |\  \overline C ^t H C = H\}\subset SL(2, \hh).$$

\begin{theorem}
An element $g \in \mathbb{G}$ is a classical M\"obius transformation of $\B$ if and only if $g(q)=(qc+d)^{-1} \cdot (qa+b)$ with
$\begin{bmatrix}
a & c\\
b & d                                
\end{bmatrix}\in Sp(1,1)$. This is equivalent to
\begin{equation*}
g(q) = v^{-1} (1-q\bar q_0)^{-1}(q-q_0)u
\end{equation*}
for some $u,v \in \partial \B, q_0 \in \B$.
\end{theorem}

For a proof, see \cite{poincare}. Hence, $\mathbb{M}$ is anti-isomorphic to $Sp(1,1)/\{\pm \I\}$.
Since $\mathbb{G}$ leaves invariant the family $\mathcal{F}_1$ of circles and affine lines of $\mathbb{H}$, and since the elements of $\mathbb{G}$ are conformal, the group $\mathbb{M}$ of classical M\"obius transformations of $\B$ preserves the following class of curves.
\begin{definition}\label{non-euclid}
If $q_1 \neq q_2 \in \B$ are $\mathbb{R}-$linearly dependent, then the diameter of $\B$ through $q_1,q_2$ is called the \emph{non-Euclidean line} through $q_1$ and $q_2$. Otherwise, the \emph{non-Euclidean line} through $q_1$ and $q_2$ is defined as the unique circle through $q_1,q_2$ that intersects $\partial \mathbb{B}= \s^3$ orthogonally.
\end{definition}

\begin{theorem}
The formula
\begin{equation}\label{poincaredist}
\delta_{\mathbb{B}}(q_1,q_2)
=\frac{1}{2}\log\left(\frac{1+|1-q_1\overline{q_2}|^{-1}|q_1-q_2|}{1-|1-q_1\overline{q_2}|^{-1}|q_1-q_2|}\right)
\end{equation}
(for $q_1,q_2 \in\B$) defines a distance that has the non-Euclidean lines as its geodesics. The elements of $\mathbb{M}$ and the map $q \mapsto \overline{q}$ are all isometries for $\delta_{\mathbb{B}}$. 
\end{theorem}

We refer the reader to \cite{ahlfors1988}; a detailed presentation can be found in \cite{poincare}.

So far, we mentioned properties of the classical M\"obius transformations that are completely analogous to the complex case. However, the analogy fails on one crucial point. The group $\mathbb{M}$ is not included in the best known analog of the class of holomorphic functions: the set of Fueter regular functions, i.e., the kernel of $\frac{\partial}{\partial x_0}+i\frac{\partial}{\partial x_1}+ j\frac{\partial}{\partial x_2}+k\frac{\partial}{\partial x_3}$ (see \cite{sudbery}). For instance, none of the rotations $q \mapsto aq$ with $a \in \hh, a \neq 0$ is Fueter regular, nor are any of the transformations listed as (i),(ii),(iii),(iv) in our previous discussion. The variant of the Fueter class considered in \cite{perotti2009}, defined as the kernel of $\frac{\partial}{\partial x_0}+i\frac{\partial}{\partial x_1}+ j\frac{\partial}{\partial x_2}-k\frac{\partial}{\partial x_3}$, includes part of the group, for instance the rotations $q \mapsto qb$ for $b \in \hh, b \neq 0$, but not all of it (for instance $q \mapsto kq$ is not in the kernel, nor is $q \mapsto q^{-1}$). 

A more recent theory of quaternionic functions, introduced in \cite{cras,advances}, has proven to be more comprehensive. The theory is based on the next definition.

\begin{definition}
Let $\Omega$ be a domain in $\hh$ and let $f : \Omega \to \hh$ be a function. For all $I \in \s = \{q \in \hh \ |\  q^2 = -1\}$, let us denote $L_I = \rr + I \rr$, $\Omega_I = \Omega \cap L_I$ and $f_I = f_{|_{\Omega_I}}$. 
The function $f$ is called (Cullen or) \emph{slice regular} if, for all $I \in \s$, the function $\bar \partial_I f : \Omega_I \to \hh$ defined by
$$
\bar \partial_I f (x+Iy) = \frac{1}{2} \left( \frac{\partial}{\partial x}+I\frac{\partial}{\partial y} \right) f_I (x+Iy)
$$
vanishes identically.
\end{definition}

By direct computation, the class of slice regular functions includes all of the generators we listed as (i),(ii),(iii),(iv). It does not contain the whole group $\mathbb{G}$ (nor its subgroup $\mathbb{M}$), because composition does not, in general, preserve slice regularity. However, \cite{moebius} introduced the new classes of \emph{(slice) regular fractional transformations} and \emph{(slice) regular M\"obius transformations} of $\B$, which are nicely related to the classical linear fractional transformations. They are presented in detail in sections \ref{regularfractional} and \ref{regularmoebius}.

One of the purposes of the present paper is, in fact, to compare the slice regular fractional transformations with the classical ones. Furthermore, we take a first glance at the role played by slice regular M\"obius transformations in the geometry of $\B$. In section \ref{differential}, we undertake a first study of their differential properties: we prove that they are not, in general, conformal, and that they do not preserve the Poincar\'e distance $\delta_\B$. In section \ref{schwarzpick}, we overview the main results of \cite{schwarzpick}. Among them is a quaternionic analog of the Schwarz-Pick lemma, which discloses the possibility of using slice regular functions in the study of the intrinsic geometry of $\B$.


\section{Regular fractional transformations}\label{regularfractional}

This section surveys the algebraic structure of slice regular functions, and its application to the construction of regular fractional transformations. From now on, we will omit the term `slice' and refer to these functions as regular, {\it tout court}.
Since we will be interested only in regular functions on Euclidean balls
$$B(0,R) = \{q \in \hh \ |\  |q| <R\},$$
or on the whole space $\hh = B(0,+\infty)$, we will follow the presentation of \cite{zeros,poli}. However, we point out that many of the results we are about to mention have been generalized to a larger class of domains in \cite{advancesrevised}.

\begin{theorem} 
Fix $R$ with $0<R\leq + \infty$ and let $\mathcal{D}_R$ be the set of regular functions $f:B(0,R)\to \hh$. Then $\mathcal{D}_R$  coincides with the set of quaternionic power series $f(q) =\sum_{n \in \nn} q^n a_n$ (with $a_n \in \hh$) converging in $B(0,R)$. Moreover, defining the \emph{regular multiplication} $*$ by the formula
\begin{equation}
\left(\sum_{n \in \nn} q^n a_n\right) *  \left(\sum_{n \in \nn} q^n b_n\right)= \sum_{n \in \nn} q^n \sum_{k=0}^n a_k b_{n-k},\end{equation}
we conclude that $\mathcal{D}_R$ is an associative real algebra with respect to $+,*$.
\end{theorem}

The ring $\mathcal{D}_R$ admits a classical ring of quotients 
$$\mathcal{L}_R = \{f^{-*}*g \ |\  f,g \in \mathcal{D}_R, f \not \equiv 0\}.$$ 
In order to introduce it, we begin with the following definition.

\begin{definition}\label{conjugate}
Let $f(q) = \sum_{n \in \nn} q^n a_n$ be a regular function on an open ball $B = B(0,R)$. The \textnormal{regular conjugate} of $f$, $f^c : B \to \hh$, is defined as $f^c(q) = \sum_{n \in \nn} q^n \bar a_n$ and the \textnormal{symmetrization} of $f$, as $f^s = f * f^c = f^c*f$. 
\end{definition}

Notice that $f^s(q) = \sum_{n \in \nn} q^n r_n$ with $r_n = \sum_{k = 0}^n a_k \bar a_{n-k} \in \rr$. Moreover, the zero-sets of $f^c$ and $f^s$ have been fully characterized.

\begin{theorem}\label{conjugatezeros}
Let $f$ be a regular function on $B = B(0,R)$. For all $x,y \in \rr$ with $x+y\s \subseteq B$, the regular conjugate $f^c$ has as many zeros as $f$ in $x+y\s$. Moreover, the zero set of the symmetrization $f^s$ is the union of the $x+y\s$ on which $f$ has a zero.
\end{theorem}

We are now ready for the definition of regular quotient. We denote by 
$$\z_h = \{q \in B \ |\  h(q) = 0\}$$ 
the zero-set of a function $h$.

\begin{definition}\label{quotient}
Let $f,g : B = B(0,R) \to \hh$ be regular functions. The \emph{left regular quotient} of $f$ and $g$ is the function $f^{-*} * g$ defined in $B \setminus \z_{f^s}$ by 
\begin{equation}
f^{-*} * g (q) = {f^s(q)}^{-1} f^c * g(q).
\end{equation} 
Moreover, the \emph{regular reciprocal} of $f$ is the function $f^{-*} = f^{-*} * 1$.
\end{definition}

Left regular quotients proved to be regular in their domains of definition. 
If  we set $(f^{-*}*g)*(h^{-*}*k) = (f^{s}h^{s})^{-1} f^c*g*h^c*k$ then $(\mathcal{L}_R,+,*)$ is a division algebra over $\mathbb{R}$ and it is the classical ring of quotients of $(\mathcal{D}_R,+,*)$ (see \cite{rowen}). In particular, $\mathcal{L}_R$ coincides with the set of \emph{right regular quotients} 
$$g*h^{-*} (q) = {h^s(q)}^{-1} g * h^c(q).$$ 
The definition of regular conjugation and symmetrization is extended to $\mathcal{L}_R$ setting $(f^{-*}*g)^c = g^c*(f^c)^{-*}$ and $(f^{-*}*g)^s(q) = {f^s(q)}^{-1}g^s(q)$.
Furthermore, the following relation between the left regular quotient $f^{-*} * g(q)$ and the quotient $f(q)^{-1} g (q)$ holds. 

\begin{theorem}\label{quotients}
Let $f,g$ be regular functions on $B=B(0,R)$. Then
\begin{equation}
f*g(q) = f(q) g(f(q)^{-1}qf(q)),
\end{equation}
and setting $T_f(q) = f^c(q)^{-1} q f^c(q)$ for all $q \in B \setminus \z_{f^s}$,
\begin{equation}
f^{-*}*g(q) = f(T_f(q))^{-1} g(T_f(q)),
\end{equation}
for all $q \in B \setminus \z_{f^s}$. For all $x,y \in \rr$ with $x+y\s\subset B \setminus \z_{f^s}$, the function $T_f$ maps $x+y\s$ to itself (in particular $T_f(x) = x$ for all $x \in \rr$). Furthermore, $T_f$ is a diffeomorphism from $B \setminus \z_{f^s}$ onto itself, with inverse $T_{f^c}$.
\end{theorem}

We point out that, so far, no simple result relating $g*h^{-*}(q)$ to $g(q) h(q)^{-1}$ is known.

This machinery allowed the introduction in \cite{moebius} of regular analogs of linear fractional transformations, and of M\"obius transformations of $\B$. To each 
$A=\begin{bmatrix}
a & c \\
b & d
\end{bmatrix} \in GL(2, \hh)$ we can associate the \emph{regular fractional transformation} 
$$\f_A (q) = (qc+d)^{-*}*(qa+b).$$
By the formula $(qc+d)^{-*}*(qa+b)$ we denote the aforementioned left regular quotient $f^{-*}*g$ of $f(q) = qc+d$ and $g(q) = qa +b$. We denote the $2 \times 2$ identity matrix as $\I$.

\begin{theorem}
Choose $R>0$ and consider the ring of quotients of regular quaternionic functions in $B(0,R)$, denoted by $\mathcal{L}_R$.
Setting 
\begin{equation}
f.A = (f c + d)^{-*}*(f a+ b)
\end{equation} 
for all $f \in \mathcal{L}_R$ and for all $A =
\begin{bmatrix}
a & c \\
b & d
\end{bmatrix} \in GL(2,\hh)$
defines a right action of $GL(2,\hh)$ on $\mathcal{L}_R$. A left action of $GL(2,\hh)$ on $\mathcal{L}_R$ is defined setting
\begin{equation}
^{t}A.f=(a*f+b)*(c*f+d)^{-*}.
\end{equation} 
The stabilizer of any element of $\mathcal{L}_R$ with respect to either action includes the normal subgroup $N=\left \{t \cdot \I\ |\  t \in \rr \setminus \{0\} \right \}\unlhd GL(2,\hh)$. Both actions are faithful when reduced to $PSL(2,\hh) = GL(2,\hh)/N$.
\end{theorem}

The statements concerning the right action are proven in \cite{moebius}, the others can be similarly derived. The two actions coincide in one special case.

\begin{proposition}\label{hermitian}
For all Hermitian matrices $A = \begin{bmatrix}
a & \bar b \\
b & d
\end{bmatrix}$ with $a,d \in \rr, b \in \hh$, 
$$f.A = (f \bar b + d)^{-*}*(f a+ b) =(a*f+b)*(\bar b*f+d)^{-*}= A^{t}.f$$
\end{proposition}

\begin{proof}
We observe that
$$(f \bar b + d)^{-*}*(f a+ b) =(a*f+b)*(\bar b*f+d)^{-*}$$
if, and only if,
$$(f a+ b)*(\bar b*f+d) =(f \bar b + d)*(a*f+b),$$
which is equivalent to
$$a f* \bar b*f + |b|^2 f+ ad f + db = a f* \bar b*f +ad f+ |b|^2 f +db.$$
\end{proof}

In general, a more subtle relation holds between the two actions.

\begin{remark}\label{leftright}
For all $A \in GL(2,\hh)$ and for all $f\in \mathcal{L}_R$, by direct computation 
\begin{equation*}
\left(f.A\right)^c=\bar{A}^t.f^c
\end{equation*}
As a a consequence, if $A\in GL(2, \hh)$ is Hermitian then $(f.A)^c= f^c.\bar{A}$.
\end{remark}

Interestingly, neither action is free, not even when reduced to $PSL(2,\hh)$. Indeed, the stabilizer of the identity function with respect to either action of $GL(2,\hh)$ equals
$$\{c\cdot \I\ |\  c \in \hh \setminus \{0\} \},$$
a subgroup of $GL(2,\hh)$ that strictly includes $N$ and is not normal. As a consequence, the set of regular fractional transformations
$$\rft = \{\f_A \ |\  A \in GL(2,\rr)\} = \{\f_A \ |\  A \in SL(2,\rr)\},$$ 
which is the orbit of the identity function $id = \f_{\I}$ under the right action of $GL(2,\hh)$ on $\mathcal{L}_\infty$, does not inherit a group structure from $GL(2,\hh)$.

\begin{lemma}
The set $\rft$ of regular fractional transformations is also the orbit of $id$ with respect to the left action of $GL(2,\hh)$ on $\mathcal{L}_\infty$.
\end{lemma}

\begin{proof}
Let $A=\begin{bmatrix}
a & c \\
b & d
\end{bmatrix} \in GL(2,\hh)$, and let us prove that $\f_A = id.A$ can also be expressed as $C.id$ for some $C \in GL(2,\hh)$. If $c=0$ then 
$$\f_A(q) = d^{-1}*(qa+b) = (d^{-1}a)*q+ d^{-1}b,$$ 
else 
$$\f_A(q) = \f_{c^{-1}A}(q) = (q-p)^{-*}*(q\alpha +\beta) = [(q-p)^s]^{-1} (q-\bar p)*(q\alpha + \beta)$$
for some $p, \alpha, \beta \in \hh$. If  $p= x+Iy$ then there exists $\tilde p \in x+y\s$ and $\gamma, \delta \in \hh$ such that $ (q-\bar p)*(q\alpha + \beta) = (q \gamma + \delta)*(q-\tilde p)$; additionally, $(q-p)^s = (q-\tilde p)^s$ (see \cite{zeros} for details). Hence,
$$\f_A(q) = [(q-\tilde p)^s]^{-1} (q \gamma + \delta)*(q-\tilde p) = (q \gamma + \delta) * (q-\overline{\tilde p})^{-*} =  (\gamma* q + \delta) * (q-\overline{\tilde p})^{-*},$$
which is of the desired form.
Similar manipulations prove that for all $C \in GL(2,\hh)$, the function $C.id$ equals $\f_A = id.A$ for some $A \in GL(2,\hh)$.
\end{proof}

We now state an immediate consequence of the previous lemma and of remark \ref{leftright}.

\begin{remark}
The set $\rft$ of regular fractional transformations is preserved by regular conjugation.
\end{remark}


\section{Regular M\"obius transformations of $\B$}\label{regularmoebius}

The regular fractional transformations that map the open quaternionic unit ball $\B$ onto itself, called \emph{regular M\"obius transformations  of $\B$}, are characterized by two results of \cite{moebius}.  
\begin{theorem}
For all $A \in SL(2, \hh)$, the regular fractional transformation $\f_A$ maps $\B$ onto itself if and only if $A \in Sp(1,1)$, if and only if there exist (unique) $u \in \partial \B, a \in \B$ such that
\begin{equation}
\f_A(q) = (1-q \bar a)^{-*}*(q-a)u.
\end{equation}
In particular, the set $\mathfrak{M}= \{f \in \rft \ |\  f(\B) = \B\}$ of the regular M\"obius transformations of $\B$ is the orbit of the identity function under the right action of $Sp(1,1)$.
\end{theorem}

\begin{theorem}
The class of regular bijective functions $f : \B \to \B$ coincides with the class $\mathfrak{M}$ of regular M\"obius transformations of $\B$.
\end{theorem}

As a consequence, the right action of $Sp(1,1)$ preserves the class of regular bijective functions from $\B$ onto itself.
We now study, more in general, the effect of the actions of $Sp(1,1)$ on the class
$$\reg = \{f:\B \to \B\ |\ f \mathrm{\ is\ regular}\}.$$

\begin{proposition}
If $f \in \reg$ then for all $a \in \B$
\begin{equation}
(1-f \bar a)^{-*}*(f-a) =(f-a)*(1-\bar a*f)^{-*}.
\end{equation}
Moreover, the left and right actions of $Sp(1,1)$ preserve $\reg$.
\end{proposition}

\begin{proof}
The fact that $(1-f\bar a)^{-*}*(f-a)=(f-a)*(1-\bar a*f)^{-*}$ is a consequence of proposition \ref{hermitian}.

Let us turn to the second statement, proving that for all $a \in \B, u,v \in \partial \B$, the function 
$$\tilde f =v^{-1}*(1-f\bar a)^{-*}*(f-a) u = (v-f\bar av)^{-*}*(f-a) u$$ 
is in $\reg$. The fact that for all $a \in \B, u,v \in \partial \B$ the function $u*(f-a)*(1-\bar a*f)^{-*} *v^{-*}$ belongs to $\reg$ will then follow from the equality just proven.

The function $\tilde f$ is regular in $\B$ since $h=v- f \bar av$ has no zero in $\B$ (as a consequence of the fact that $|a|<1$, $|f|<1$ and $|v|=1$). Furthermore,
$$\tilde f = (v- (f \circ T_h) \bar av)^{-1}(f\circ T_h-a)u=$$
$$v^{-1}(1- (f \circ T_h) \bar a)^{-1}(f\circ T_h-a)u =v^{-1}(M_a \circ f \circ T_h)u$$
where $T_h$ and $M_a(q) = (1-q \bar a)^{-1}(q-a)$ map $\B$ to itself bijectively and $u,v \in \partial\B$.
Hence, $\tilde f=v^{-*}*(1-f\bar a)^{-*}*(f-a) *u \in \reg$, as desired.
\end{proof}

As a byproduct, we obtain that the orbit of the identity function under the left action of $Sp(1,1)$ equals $\mathfrak{M}$.

\begin{proposition}
If $f\in \reg$ then its regular conjugate $f^c$ belongs to $\reg$ as well. Furthermore, $f^c$ is bijective (hence an element of $\mathfrak{M}$) if and only if $f$ is.
\end{proposition}

\begin{proof}
Suppose $f^c(p) = a \in \hh \setminus \B$ for some $p = x+Iy \in \B$. Then $p$ is a zero of the regular function $f^c-a$. By theorem \ref{conjugatezeros}, there exists $\tilde p \in x+y\s \subset \B$ such that $(f^c-a)^c = f-\bar a$ vanishes at $\tilde p$. Hence, $f(\B)$ includes $\bar a \in \hh \setminus\B$, a contradiction with the hypothesis $f(\B) \subseteq \B$.

As for the second statement, $f^c$ is bijective if and only if, for all $a \in \B$, there exists a unique $p \in \B$ such that $f^c(p) =a$. Reasoning as above, we conclude that this happens if and only if for all $a \in \B$, there exists a unique $\tilde p \in \B$ such that $f(\tilde p) = \bar a$. This is equivalent to the bijectivity of $f$.
\end{proof}


\section{Differential and metric properties of regular M\"obius transformations}\label{differential}

The present section is concerned with two natural questions:  
\begin{enumerate}
\item[(a)] whether the regular M\"obius transformations are conformal (as the classical M\"obius transformations); 
\item[(b)] whether they preserve the quaternionic Poincar\'e distance defined on $\B$ by formula (\ref{poincaredist}).
\end{enumerate}
For a complete description of the Poincar\'e metric, see \cite{poincare}. In order to answer question (a), we will compute for a regular M\"obius transformation the series development introduced by the following result of \cite{expansion}. Let us set
$$U(x_0+y_0\s,r) = \{q \in \hh \ |\  |(q-x_0)^2+y_0^2| < r^2\}$$
for all $x_0,y_0 \in \rr$, $r>0$.

\begin{theorem}
Let $f$ be a regular function on $\Omega=B(0,R)$, and let  $U(x_0+y_0\s,r) \subseteq \Omega$. Then for each $q_0 \in x_0+y_0 \s$ there exists $\{A_n\}_{n \in \nn} \subset \hh$ such that
\begin{equation}\label{expansion}
f(q) = \sum_{n \in \nn} [(q-x_0)^2+y_0^2]^n [A_{2n}+(q-q_0) A_{2n+1}]
\end{equation}
for all $q \in U(x_0+y_0\s,r)$. As a consequence,
\begin{equation}
\lim_{t\to 0} \frac{f(q_0+tv)-f(q_0)}{t} = v A_1 + (q_0v - v\bar q_0)A_2.
\end{equation}
for all $v \in T_{q_0}\Omega\cong \hh$.
\end{theorem}

If $q_0 \in L_I$ and if we split the tangent space $T_{q_0}\Omega\cong \hh$ as $\hh = L_I \oplus L_I^\perp$, then the differential of $f$ at $q_0$ acts by right multiplication by $A_1$ on $L_I^\perp$ and by right multiplication by $A_1+2Im(q_0)A_2$ on $L_I$.
Furthermore, if for all $q_0 \in \Omega$ the \emph{remainder} $R_{q_0}f$ is defined as
$$R_{q_0}f(q) = (q-q_0)^{-*}*(f(q)-f(q_0))$$
then the coefficients of \eqref{expansion} are computed as $A_{2n} = (R_{\bar q_0}R_{q_0})^nf(q_0)$ and $A_{2n+1}= R_{q_0}(R_{\bar q_0}R_{q_0})^nf(\bar q_0)$.

Let us recall the definition of the \emph{Cullen derivative} $\partial_cf$, given in \cite{advances} as 
\begin{equation}\label{cullen}
\partial_cf(x+Iy)=\frac{1}{2}\left(\frac{\partial}{\partial x}-I\frac{\partial}{\partial y}\right)f(x+Iy)
\end{equation}
for $I \in \s,\ x,y \in \rr$, as well as the definition of the \emph{spherical derivative} 
\begin{equation}\label{spherical}
\partial_sf(q) = (2Im(q))^{-1}(f(q)-f(\bar q))
\end{equation}
given in \cite{perotti}. We can make the following observation.

\begin{remark}
If $f$ is a regular function on $B(0,R)$ and if \eqref{expansion} holds then $\partial_cf(q_0) = R_{q_0}f(q_0) = A_1+2Im(q_0)A_2$ and $\partial_sf(q_0)= R_{q_0}f(\bar q_0) = A_1$.
\end{remark}

In the case of the regular M\"obius transformation
$$\mr_{q_0}(q) = (1-q\bar q_0)^{-*}*(q-q_0) = (q-q_0)*(1-q\bar q_0)^{-*},$$
clearly $R_{q_0}\mr_{q_0}(q) = (1-q\bar q_0)^{-*}$, so that we easily compute the coefficients $A_n$.

\begin{proposition}
Let $q_0 = x_0+y_0I \in \B$. Then the expansion \eqref{expansion} of $\mr_{q_0}$ at $q_0$ has coefficients
\begin{eqnarray}
A_{2n-1} &=& \frac{\bar q_0^{2n-2}}{(1-|q_0|^2)^{n-1}(1-\bar q_0^2)^{n}}\\
A_{2n} &=& \frac{\bar q_0^{2n-1}}{(1-|q_0|^2)^n(1-\bar q_0^2)^n}
\end{eqnarray}
for all $n\geq 1$. As a consequence, for all $v \in \hh$,
\begin{equation}
\frac{\partial\mr_{q_0}} {\partial v}(q_0)=v (1-\bar q_0^2)^{-1}+(q_0 v-v \bar q_0) \frac{\bar q_0}{(1-|q_0|^2)(1-\bar q_0^2)}.
\end{equation}
\end{proposition}

\begin{proof}
We have already observed that $R_{q_0}\mr_{q_0}(q) = (1-q\bar q_0)^{-*}$, so that $A_1 = R_{q_0}\mr_{q_0}(\bar q_0) = (1-\bar q_0^2)^{-1}$. Moreover,
$$R_{\bar q_0}R_{q_0}\mr_{q_0}(q) = (q-\bar q_0)^{-*}*\left[R_{q_0}\mr_{q_0}(q)- A_{1}\right]=$$
$$=(q-\bar q_0)^{-*}*\left[(1-q\bar q_0)^{-*} - A_{1}\right] =$$
$$= (1-q\bar q_0)^{-*}* (q-\bar q_0)^{-*}*\left[(1-\bar q_0^2)- (1-q\bar q_0)\right]A_1=$$
$$= (1-q\bar q_0)^{-*} \bar q_0 A_1.$$
The thesis follows by induction, proving that for all $n \geq 1$
$$(R_{\bar q_0}R_{q_0})^{n}\mr_{q_0}(q) =(1-q\bar q_0)^{-*} \bar q_0 A_{2n-1},$$
$$R_{q_0}(R_{\bar q_0}R_{q_0})^{n}\mr_{q_0}(q) = (1-q\bar q_0)^{-*} \bar q_0 A_{2n}$$
by means of similar computations.
\end{proof}

We are now in a position to answer question (a).

\begin{remark}
For each $q_0 = x_0 + I y_0\in \B \setminus \rr$, the differential of $\mr_{q_0}$ at $q_0$ acts by right multiplication by $\partial_c\mr_{q_0}(q_0)=(1-|q_0|^2)^{-1}$ on $L_I$ and by right multiplication by $\partial_s\mr_{q_0}(q_0)=(1-\bar q_0^2)^{-1}$  on $L_I^\perp$. Since these quaternions have different moduli, $\mr_{q_0}$ is not conformal.
\end{remark}

We now turn our attention to question (b): whether or not regular M\"obius transformations preserve the quaternionic Poincar\'e metric on $\B$ described in section \ref{classicM} and in \cite{poincare}. We recall that this metric was constructed to be preserved by the classical (non regular) M\"obius transformations of $\B$. Thanks to theorem \ref{quotients}, we observe what follows.

\begin{remark}
If $\mr_{q_0}(q) = (1-q\bar q_0)^{-*}*(q-q_0)$ and $M_{q_0}(q) = (1-q\bar q_0)^{-1}(q-q_0)$ then
\begin{equation}
\mr_{q_0}(q)= M_{q_0}(T(q)).
\end{equation}
where $T(q) = (1-q q_0)^{-1} q (1-q q_0)$ is a diffeomorphism of $\B$ with inverse $T^{-1}(q) = (1-q \bar q_0)^{-1} q (1- q \bar q_0)$.
\end{remark}

Thus, a generic regular M\"obius transformation of $\B$
$$q \mapsto \mr_{q_0}(q) u = M_{q_0}(T(q)) u$$ 
(with $u \in \partial \B$) is an isometry if and only if $T$ is. An example shows that this is not the case.

\begin{example}
Let $q_0=\frac{I_0}{2}$ for some $I_0 \in \s$. Then $T(q) = (1-q q_0)^{-1} q (1-q q_0)$ is not an isometry for the Poincar\'e metric defined by formula \eqref{poincaredist}. Indeed, if $J_0 \in \s, J_0 \perp I_0$ then setting $q_1=\frac{J_0}{2}$ we have
$$\delta_\B(q_0,q_1) > \delta_\B(T(q_0),T(q_1))$$
since
$$\frac{|q_1-q_0|^2}{|1-q_1\bar q_0|^2} = \frac{|2J_0-2I_0|^2}{|4+J_0I_0|^2} =\frac{8}{17}$$
while, computing $T(q_0) = q_0 = \frac{I_0}{2}$ and $T(q_1) = \frac{8I_0+15J_0}{34}$, we conclude that
$$\frac{|T(q_1)-q_0|^2}{|1-T(q_1)\bar q_0|^2} = 4\frac{|-3I_0+5J_0|^2}{|20-5K_0|^2}= \frac{8}{25}.$$
Thus, the regular M\"obius transformations do not have a definite behavior with respect to $\delta_\B$: we have proven that $T$ (hence $\mr_{q_0}$) is not an isometry, nor a dilation; the same computation proves that  $T^{-1}(q) = (1-q \bar q_0)^{-1} q (1- q \bar q_0)$ (hence $\mr_{\bar q_0}$) is not a contraction.
\end{example}

The previous discussion proves that the study of regular M\"obius transformations cannot be framed into the classical study of $\B$, and that it requires further research. On the other hand, the theory of regular functions provides working tools that were not available for the classical M\"obius transformations. These tools lead us in \cite{schwarzpick} to an analog of the Schwarz-Pick lemma, which we will present in the next section.


\section{The Schwarz-Pick lemma for regular functions}\label{schwarzpick}

In the complex case, holomorphic functions play a crucial role in the study of the intrinsic geometry of the unit disc $\Delta = \{z \in \cc\ |\  |z|<1\}$ thanks to the Schwarz-Pick lemma \cite{pick2,pick1}.

\begin{theorem}
Let $f:\Delta \to \Delta$ be a holomorphic function and let $z_0 \in \Delta$. Then
\begin{equation}
\left|\frac{f(z)-f(z_0)}{1-\overline{f(z_0)}f(z)}\right| \leq \left|\frac{z-z_0}{1-\bar z_0 z}\right|,
\end{equation}
for all $z \in \Delta$ and
\begin{equation}
\frac{|f'(z_0)|}{1-|f(z_0)|^2} \leq \frac{1}{1-|z_0|^2}.
\end{equation}
All inequalities are strict for $z \in \Delta \setminus \{z_0\}$, unless $f$ is a M\"obius transformation.
\end{theorem}

This is exactly the type of tool that is not available, in the quaternionic case, for the classical M\"obius transformations. To the contrary, an analog of the Schwarz-Pick lemma is proven in \cite{schwarzpick} for quaternionic regular functions.
To present it, we begin with a result concerning the special case of a function $f \in \reg$ having a zero.

\begin{theorem} \label{teo1}
If $f : \B \to \B$ is regular and if $f(q_0)=0$ for some $q_0 \in \B$, then
\begin{equation}
|\mr_{q_0}^{-*}*f(q)| \leq 1
\end{equation}
for all $q \in \B$. The inequality is strict, unless $\mr_{q_0}^{-*}*f(q) \equiv u$ for some $u \in \partial \B$.
\end{theorem}
 
A useful property of the moduli of regular products is also proven in \cite{schwarzpick}:

\begin{lemma}\label{modulusproduct}
Let $f,g,h: B(0,R) \to \hh$ be regular functions. If $|f|\leq |g|$ then $|h*f| \leq |h*g|$. Moreover, if $|f|< |g|$ then $|h*f| < |h*g|$ in $B \setminus \z_h$.
\end{lemma}

The property above allows us to derive from theorem \ref{teo1} the perfect analog of the Schwarz-Pick lemma in the special case $f(q_0)=0$. We recall that $\partial_c f$ denotes the Cullen derivative of $f$, defined by formula \eqref{cullen}, while $\partial_s f$ denotes the spherical derivative, defined by formula \eqref{spherical}.

\begin{corollary}\label{schwarzp}
If $f : \B \to \B$ is regular and if $f(q_0)=0$ for some $q_0 \in \B$ then
\begin{equation}
|f(q)| \leq |\mr_{q_0}(q)|
\end{equation}
for all $q \in \B$. The inequality is strict at all $q \in \B \setminus \{q_0\}$, unless there exists $u \in \partial \B$ such that $f(q) = \mr_{q_0}(q) \cdot u$ at all $q \in \B$.
Moreover, $|R_{q_0}f(q)| \leq |(1-q\bar q_0)^{-*}|$ in $\B$ and in particular
\begin{eqnarray}
|\partial_c f (q_0)| \le \frac{1}{1-|q_0|^2}\\
|\partial_s f(q_0)|
\leq \frac{1}{|1-\overline{q_0}^2|}.
\end{eqnarray}
These inequalities are strict, unless $f(q) = \mr_{q_0}(q) \cdot u$ for some $u \in \partial \B$.
\end{corollary}

Finally, the desired result is obtained in full generality.

\begin{theorem}[Schwarz-Pick lemma] \label{MainSchwarzPick}
Let $f : \B \to \B$ be a regular function and let $q_0 \in \B$. Then
\begin{equation}\label{MainSchwarzPickEq}
|(f(q)-f(q_0))*(1-\overline{f(q_0)}*f(q))^{-*}| \leq |(q-q_0)*(1-\bar q_0*q)^{-*}|,
\end{equation}
\begin{equation}
|R_{q_0}f(q)*(1-\overline{f(q_0)}*f(q))^{-*}| \leq |(1-\bar q_0*q)^{-*}|
\end{equation}
in $\B$. In particular,
\begin{equation} \label{culderiv}
|\partial_c f *(1-\overline{f(q_0)}*f(q))^{-*}|_{|_{q_0}} \leq \frac{1}{1-|q_0|^2}.
\end{equation}
Apart from \eqref{MainSchwarzPickEq} at $q_0$ (which reduces to $0\leq0$), all inequalities are strict unless $f$ is a regular M\"obius transformation.
\end{theorem}

This promising result makes it reasonable to expect that regular functions play an important role in the intrinsic geometry of the quaternionic unit ball. Therefore, it encourages to continue the study of regular M\"obius transformations and of their differential or metric properties.


\bibliography{BisiStoppatoArXiv}

\bibliographystyle{abbrv}


\end{document}